\documentclass[a4,twocolum]{autart}
\usepackage{graphics,color}
\usepackage{graphicx}
\usepackage{amssymb,amsmath,amsfonts}
\usepackage{amsmath,mathrsfs,bm,url,times}
\usepackage{latexsym}
\usepackage{subfigure}
\usepackage{algorithmic}
\usepackage{algorithm}
\usepackage{natbib}

\newtheorem{proposition}{Proposition}
\newtheorem{theorem}{Theorem}
\newtheorem{lemma}{Lemma}
\newtheorem{definition}{Definition}

\newtheorem{remark}{Remark}
\newtheorem{corollary}{Corollary}

\newenvironment{proof}{\vspace{1ex}\noindent{\bf Proof}\hspace{0.5em}}
	{\hfill\qed\vspace{1ex}}
\def\qed{\hfill \vrule height 6pt width 6pt depth 0pt}

\newcommand{\R}{{\mathbb R}}
\newcommand{\onetom}{1,\cdots,m}
\newcommand{\oneton}{1,\cdots,n}
\newcommand{\onetoK}{1,\cdots,K}

\newcommand{\T}{\mathcal T}

\newcommand{\ignore}[1]{}

\newcommand{\CASE}[1]{\STATE \textbf{case} #1\textbf{:} \begin{ALC@g}}
\newcommand{\ENDCASE}{\end{ALC@g}}

\newcommand{\DEFAULT}{\STATE \textbf{default:} \begin{ALC@g}}
\newcommand{\ENDDEFAULT}{\end{ALC@g}}
\newcommand{\DEFAULTLINE}[1]{\STATE \textbf{default:} }

\begin{document}
\begin{frontmatter}

\title{Contraction and incremental stability of switched Carath\'{e}odory systems using multiple norms\thanksref{foot1}}
\thanks[foot1]{This work is jointly supported by the National Natural Sciences Foundation of China under Grant No. 61273309, the Program for New Century Excellent Talents in University (NCET-13-0139), the Key Laboratory of Nonlinear Science of Chinese
Ministry of Education and the Shanghai Key Laboratory for Comtemporary Applied Mathematics, Fudan University.\\
Corresponding Author: W. L. Lu. Tel: +86-21-65643265. Fax: +86-21-65646073.}
\author[Shanghai]{Wenlian Lu}\ead{wenlian@fudan.edu.cn}\and
\author[Naples]{Mario di Bernardo}\ead{mario.dibernardo@unina.it}
\address[Shanghai]{School of Mathematical Sciences and Centre for Computational Systems Biology, Fudan University, China}
\address[Naples]{Department of Electrical Engineering and Information Technology, University of Naples Federico II,  Italy,\\ Department of Engineering Mathematics, University of Bristol,  U.K.}
\begin{abstract}
In this paper, incremental exponential asymptotic stability of a class of switched Carath\'{e}odory nonlinear systems is studied based on the novel concept of measure of switched matrices via multiple norms and the transaction coefficients between these norms. This model is rather general and includes the case of staircase switching signals as a special case. Sufficient conditions are derived for incremental stability allowing for the system to be incrementally exponentially asymptotically stable even if some of its modes are unstable in some time periods. Numerical examples on switched linear systems switching periodically and on the synchronization of switched networks of nonlinear systems are used to illustrate the theoretical results.
\end{abstract}
\end{frontmatter}
\section{Introduction}
Studying incremental stability of nonlinear systems is particularly important in many application areas, including observer design and, more recently, consensus and synchronisation problems in network control where convergence analysis is a fundamental step \citep{Wan_Slo_05,Rus_diB_09,Rus_diB_09b,Rus_diB_Son_10,Rus_diB_Slo_11,RuBe:13}.

Since the early work by \cite{Lewis,de:67}, contraction theory has been highlighted as a promising approach to study incremental exponential asymptotic stability ($\delta$EAS) of nonlinear systems \citep{Loh_Slo_98,FoSe:14,Ang}; also see \cite{Jou_05} for an historical overview. In particular, as shown by \cite{Loh_Slo_98}, sufficient conditions for $\delta$EAS of a given nonlinear system over an invariant set of interest can be obtained by studying the matrix measure of its Jacobian induced by some vector norm. It is possible to prove, as done by \cite{Loh_Slo_98,Rus_diB_Son_10}, that if such measure is negative definite in that set for all time then any two trajectories will exponentially converge towards each other; the rate  of convergence being estimated by the negative upper bound on the Jacobian measure.

Numerous applications of contraction analysis have been presented in the literature from observer design to the synthesis of network control systems. See e.g. \cite{Loh_Slo_98,FoSe:14,Rus_diB_Slo_11}.
Remarkably, the problem of studying incremental stability of switched and hybrid systems has attracted relatively little attention despite the large number of potential applications, e.g. power electronic networks, variable structure systems, walking and hopping robots, to name just a few \citep{BeBu:08,Li:03,Co:08}.

It has been suggested by \cite{RuDi:11,DiLi:14} that extending contraction analysis to this class of systems can be a viable and effective approach to obtain conditions for their incremental asymptotic stability. Related approaches include the work on convergence of piecewise affine continuous systems presented by \cite{PaPo:05,PaPo:05_2,PaPo:07,PaWo:08} and the recent conference papers \citep{dili:12,DiFi:14}.

One limitation of the existing extensions of contraction theory to switched systems, \cite[e.g.][]{DiLi:14}, is that they rely on the use of a unique matrix measure to assess the Jacobian of each system modes. This is a particularly restricting assumption as it would be desirable to use measures induced by different norms to evaluate the Jacobian of each of the system modes. This would correspond to studying incremental stability of the switched system with multiple incremental Lyapunov functions rather than using a common one (which is much harder to find).

The aim of this paper is to address this problem and present conditions for contraction and incremental stability of a large class of switched Carath\'{e}odory systems. The key idea is to define a novel concept of matrix measure via multiple norms and exploit the transaction coefficients between these norms. In so doing, sufficient conditions are derived for incremental stability that allow for a system to be $\delta$EAS even if some of its modes are unstable (or not contracting) over some time intervals. The theoretical results are illustrated via their applications to some representative examples, including synchronisation in blinking networks.


\section{Preliminaries}\label{sec:norm}

We focus on switched dynamical systems of the form
\begin{eqnarray}
\dot{x}=f(x,r(t))\label{ds1}
\end{eqnarray}
where $x\in\R^{n}$ and the switching signal $r(t)$ is assumed to be a real-valued piecewise continuous (PWC for short) function with respect to time: there exist countable discontinuous points $t_{0}<t_{1}<\cdots<t_{i}<\cdots$ such that $r(t_{i}\pm)$ exist and $r(t_{i})=r(t_{i}+)$ for all $t_{i}$. A typical example is the staircase function $
r(t)=\xi_{i}$, for $t_{i}\le t< t_{i+1},~i=0,1,\cdots$,
for the increasing time sequence $\{t_{j}\}_{j\ge 0}$, which has been widely used as switching signal in control systems \citep{Li:03}.



Here we make the following hypothesis:
\begin{itemize}
\item[${\mathcal H}_{1}$:]  For some $C\subset\R^{n}$, the vector field $f(x,r(t)):C\times\R_{\ge 0}\to\R^{n}$ is (i) continuous with respect to $(x,r)$; (ii) continuously differentiable with respect to $x$, and (iii)  there exists a Lebesgue measurable function $m(t)$ such that $|f(x,r(t))|\le m(t)$ for all $x\in C$ and $t\in\R_{\ge 0}$.
\end{itemize}
It can be seen that under hypothesis ${\mathcal H}_{1}$, and with $r(t)$ defined as above, the vector field $f(x,r(t))$ defines a Carath\'{e}odory switched system \citep{Fi:88}. It can be proved that, given an initial condition in $C$, a solution of a Carath\'{e}dory system exists and is unique \citep{Hale}. We define $\phi(t;t_{0},x_{0},r_{t})$ as the solution of (\ref{ds1}) with $x(t_{0})=x_{0}$ and the switching signal $r(t)$, where $r_{t}$ denotes the trajectory of $r(t)$ up to $t$, i.e., $r_{t}=\{r(s)\}_{t_{0}\le s\le t}$.

In this paper, $|\cdot|_{\chi}$ stands for a specific vector norm in Euclidean space and the matrix norm induced by it, which can be defined in different ways that are  all equivalent. The {\em transaction coefficients} from the norm $|\cdot|_{a}$ to $|\cdot|_{b}$ is defined as $
\beta_{ab}=\sup_{|x|_{a}=1}|x|_{b}$ \citep{Bou}.

A continuous function $\alpha : [0, a) \to [0, \infty)$ is said to belong to {\em class $\mathcal{K}$} if (I) it is strictly increasing; (II) $\alpha(0)$ = 0. And, a continuous function $\beta (\rho,t): [0, a) \times [0, \infty) \to [0, \infty)$ is said to belong to {\em class $\mathcal{KL}$} if (1) for each fixed $t$, the function $\beta(\rho,t)$ belongs to class $\mathcal K$; (2) for each fixed $\rho$, the function $\beta(\rho,t)$ is decreasing with respect to $t$ and $\lim_{t\to\infty}\beta(\rho,t)=0$. In addition, if a function $\beta(\rho,t)$ of class $\mathcal{KL}$ converges to $0$ exponentially as $t\to\infty$, $\beta(\rho,t)$ is said to be of class ${\mathcal{EKL}}$.  Here, we give the following definition of incremental stability from \cite{Ang} with modifications.

\begin{definition}\label{def0}
System (\ref{ds1}) is said to be incrementally asymptotically stable ($\delta AS$ for short) with $r(t)$ in the region $C\subset\R^{n}$ if there exists a function $\beta(s,t)$ of class $\mathcal{KL}$ such that for any initial data $x_{0},y_{0}\in C$ and starting time $t_{0}$, the following property holds
\begin{eqnarray*}
|\phi(t+t_{0};t_{0},x_{0},r_{t})-\phi(t+t_{0};t_{0},y_{0},r_{t})|\le\beta(|x_{0}-y_{0}|,t)
\end{eqnarray*}
for some norm $|\cdot|$. If $\beta(s,t)$ is picked independently of the initial time $t_{0}$, then system (\ref{ds1}) is said to be incrementally uniformly asymptotically stable ($\delta UAS$ for short). If $\beta(s,t)$ is of class $\mathcal{EKL}$, then system (\ref{ds1}) is said to be incrementally exponentially asymptotically stable ($\delta EAS$ for short) and incrementally uniformly exponentially asymptotically stable ($\delta UEAS$ for short) if $\beta(s,t)$ is chosen independently of $t_{0}$.
\end{definition}

\begin{definition}\label{kappa}
A set $C\subset\R^{n}$ is said to be a $\kappa$-reachable set if there exists  a continuously differentiable curve $\gamma(s):[0,1]\to C$ that links $x_{0}$ and $y_{0}$, i.e., $\gamma(0)=x_{0}$ and $\gamma(1)=y_{0}$, and satisfies $
|\gamma'(s)|_{\chi(t_{0})}\le \kappa|x_{0}-y_{0}|_{\chi(t_{0})}$, for all $s\in[0,1]$ and
some constant $\kappa>0$, independently of $x_{0}$ and $y_{0}$.
\end{definition}
\section{Switched matrix measures and general contraction analysis}

The matrix measure induced by the vector norm $|\cdot|_{\chi}$, where $\chi$ is the index for the norm being used, is defined as
\begin{eqnarray*}
\mu_{\chi}(A)=\lim_{h\to 0+}\frac{1}{h}\left[|(I_{n}+hA)|_{\chi}-1\right]
\end{eqnarray*}
for a square matrix $A\in\R^{n,n}$ and was used for the contraction analysis of smooth nonlinear systems, see e.g. \cite{Loh_Slo_98}.

Given a PWC function $\chi(t)$, the left (right) limit of $|\cdot|_{\chi(t)}$ at time $t$ is defined as $\lim_{h\to 0-}|x|_{\chi(t+h)}$ ($\lim_{h\to 0+}|x|_{\chi(t+h)}$) if it exists for all $x\in R^{n}$, and is denoted by $|\cdot|_{\chi(t\pm)}$ respectively. We say that the switched norm $|\cdot|_{\chi(t)}$ is {\em continuous} at time $t$ if $|x|_{\chi(t-)}=|x|_{\chi(t+)}=|x|_{\chi(t)}$ for all $x\in\R^{n}$, i.e. if $|\cdot|_{\chi(t)}$ is {\em left and right continuous}.
We say that $|\cdot|_{\chi(t)}$ is {\em uniformly equivalent} if there exists a constant $D>0$ such that $|x|_{\chi(t)}\le D |x|_{\chi(s)}$ for all $x\in\R^{n}$ and $t,s\in \R$.

We can now extend the definition of matrix measure to the case of multiple norms, taking the time-varying nature of $\chi(t)$ into consideration, as follows.
\begin{definition}
The switched matrix measure with respect to multiple norms $|\cdot|_{\chi(t)}$ is defined as
\begin{eqnarray}
\nu_{\chi(t)}(A)=\overline{\lim_{h\to 0+}}\frac{1}{h}\sup_{|x|_{\chi(t)}=1}\left[|(I_{n}+hA)x|_{\chi(t+h)}-1\right]\label{nu}
\end{eqnarray}
if the limit exists, where $\overline{\lim}$  stands for the limit superior.
\end{definition}
It can be seen that if $\chi(t)$ is constant over an interval, say $[t,t+\delta)$, then $\nu_{\chi(t)}(A)=\mu_{\chi(t)}(A)$ over that interval.

The existence of the switched matrix measure is related to the {\em partial differential} of the switched norm $|\cdot|_{\chi(t)}$ defined as follows
\begin{eqnarray*}
\overline{\partial}_{t} (|\cdot|_{\chi(t)})=\overline{\lim_{h\to 0+}}\sup_{|x|_{\chi(t)}=1}\frac{|x|_{\chi(t+h)}-1}{h}.
\end{eqnarray*}
We say that the multiple norm $|\cdot|_{\chi(t)}$ is {\em right regular} at time $t$ if $\overline{\partial}_{t} (|\cdot|_{\chi(t)})$ exists.  Thus, we have
\begin{proposition}
If $|\cdot|_{r(t)}$ is right regular at $t$, then (i). $|\cdot|_{\chi(t)}$ is right continuous at $t$; (ii). $\nu_{\chi(t)}(A)$ exists at $t$.
\end{proposition}
\begin{proof}
The first statement is straightforward from the definitions of $\overline{\partial}_{t} (|\cdot|_{\chi(t)})$ and right continuity of $|\cdot|_{\chi(t)}$.

The quotient term of the definition of $\nu_{\chi(t)}(A)$ gives
\begin{eqnarray*}
&&\frac{1}{h}\sup_{|x|_{\chi(t)}=1}\left[|(I_{n}+hA)x|_{\chi(t+h)}-|(I_{n}+hA)x|_{\chi(t)}\right.\\
&&\left.+
|(I_{n}+hA)x|_{\chi(t)}-1\right]\\
&\le&\frac{1}{h}\sup_{|x|_{\chi(t)}=1}\left[|(I_{n}+hA)x|_{\chi(t+h)}-|(I_{n}+hA)x|_{\chi(t)}\right]
\\
&&+\frac{1}{h}\sup_{|x|_{\chi(t)}=1}\left[|(I_{n}+hA)x|_{\chi(t)}-1\right]\\
&\le&\frac{1}{h}\left[h\overline{\partial}_{t} (|\cdot|_{\chi(t)})|(I_{n}+hA)|_{\chi(t)}
+\mu_{\chi(t)}(A)h+o(h)\right]
\end{eqnarray*}
where $o(h)$ is an infinitesimal term with $\lim_{h\to 0}o(h)/h=0$. So, $
\nu_{\chi(t)}(A)\le \mu_{\chi(t)}(A)+\overline{\partial}_{t} (|\cdot|_{\chi(t)})$ (statement (ii)).
\end{proof}

Using the definition of switched matrix measure, we extend Coppel inequality \citep{Coppel} as follows.

\begin{lemma}
Suppose that $|\cdot|_{\chi(t)}$ is right-regular in $[t_{1},t_{2}]$. Consider the following time-varying Carath\'e{o}dory dynamical system $
\dot{x}(t)=A(t)x(t)$
with some PWC matrix-valued function $A(t)\in\R^{n,n}$ and $x(t)\in\R^{n}$. If $\nu_{\chi(t)}(A(t))\le \alpha(t)$ for some measurable function $\alpha(t)$ for $t\in[t_{1},t_{2}]$, then
\begin{eqnarray}
|x(t_{2})|_{\chi(t_{2})}\le\exp\left(\int_{t_{1}}^{t_{2}}\alpha(t)dt\right)|x(t_{1})|_{\chi(t_{1})}.
\label{coppel}
\end{eqnarray}
\end{lemma}
\begin{proof}
For any $t\in(t_{1},t_{2})$ except discontinuous points of $A(t)$, consider the following quotient
\begin{eqnarray*}
I(h)&=&\frac{1}{h}[|x(t+h)|_{\chi(t+h)}-|x(t)|_{\chi(t)}]\\
&=&\frac{1}{h}[|x(t)+\int_{t}^{t+h}A(s)x(s)ds|_{\chi(t+h)}-|x(t)|_{\chi(t)}]\\
&=&\frac{1}{h}\left\{|(I_{n}+hA(t))x(t)+\int_{t}^{t+h}[A(s)x(s)\right.\\
&&\left.-A(t)x(t)]ds
|_{\chi(t+h)}-|x(t)|_{\chi(t)}\right\}\\
&\le&\frac{1}{h}[|(I_{n}+hA(t))x(t)|_{\chi(t+h)}-|x(t)|_{\chi(t)}+o(h)],
\end{eqnarray*}
which implies that $\overline{\lim_{h\to 0}}I(h)\le\nu_{\chi(t)}(A(t))|x(t)|_{\chi(t)}$. That is,
$D^{+}|x(t)|_{\chi(t)}\le \nu_{\chi(t)}(A(t))|x(t)|_{\chi(t)}\le\alpha(t)|x(t)|_{\chi(t)}$ holds for almost every $t\in[t_{1},t_{2}]$, where $D^{+}$ stands for the Dini derivative. Thus, Ineq. (\ref{coppel}) holds, noting that $x(t)$ is continuous with respect to $t$.
\end{proof}

We make the following hypothesis on the multiple norm.

\begin{itemize}
\item[$\mathcal H_{2}$:] $|\cdot|_{\chi(t)}$ is right-regular everywhere but at the time instants $\{\tilde{t}_{j}\}$, it is  right-continuous and uniformly equivalent, and its left-limit exists at each $\tilde{t}_{j}$.
\end{itemize}
Thus, we denote by $\tilde{N}(t,s)=\#\{i:~s<\tilde{t}_{i}<t\}$, the number of time instants $\tilde{t}_{i}$  in the time interval $(s,t)$; obviously, $\tilde{N}(t)=\tilde{N}(t,t_{0})$. Here, $\#$ stands for the cardinality of a set.
In addition, we make the following assumption.
\begin{itemize}
\item[${\mathcal H}_{3}$:]
There exists a $\kappa$-reachable set \citep{Rus_diB_Son_10} $C\subset\R^{n}$ which is a forward-invariant for (\ref{ds1}).
\end{itemize}
Then, we have a general result on $\delta EAS$ of switched dynamical system (\ref{ds1}) by contraction analysis in multiple norms.
\begin{theorem}\label{thm0}
Suppose that hypotheses $\mathcal H_{1,2,3}$ hold and that $r(t)$ is PWC.
If there exist a measurable function $\alpha(t)$, nonnegative constants $\beta_{j}$, $j=1,2,\cdots$, $c>0$, and $T_{0}>0$, such that, for all $x\in C$, the following conditions hold
\begin{eqnarray}
&&\nu_{\chi(t)}\left(\frac{\partial f}{\partial x}(x,r(t))\right)\le \alpha(t),\forall~t\ne \tilde{t}_{j},~x\in C\label{mm}\\
&&|\cdot|_{\chi(\tilde{t}_{j})}\le \beta_{j}|\cdot|_{\chi(\tilde{t}_{j}-)},~\forall~j\label{eq}
\end{eqnarray}
and, for all $T>T_0$,
\begin{equation}
\frac{1}{T}\left[\int_{t_{0}}^{T+t_{0}}\alpha(t)dt+\sum_{j=1}^{\tilde{N}(t_{0}+T)}\log(\beta_{j})\right]<-c,
\label{cond_thm0}
\end{equation}
then system (\ref{ds1}) is $\delta EAS$ in $C$ with respect to $r(t)$; if (\ref{cond_thm0}) holds for some $c>0$ independent of $t_{0}$, then (\ref{ds1}) is $\delta EUAS$ with respect to $r(t)$ in $C$. In addition, the exponential convergence rate can be estimated as $O(\exp(-c(t-t_{0})))$.
\end{theorem}
\begin{proof}
For any $x_{0},y_{0}\in C$, $C$ being $\mathcal \kappa$-reachable for some $\kappa>0$ implies, from Definition \ref{kappa}, that there exists  a continuously differentiable curve $\gamma(s):[0,1]\to C$ such that $\gamma(0)=x_{0}$ and $\gamma(1)=y_{0}$, and $
|\gamma'(s)|_{\chi(t_{0})}\le \kappa|x_{0}-y_{0}|_{\chi(t_{0})}$, for all $s\in[0,1]$.

Let $\psi(t,s)=\phi(t;t_{0},\gamma(s),r_{t})$, $s\in[0,1]$, be the solution of (\ref{ds1}) with initial value $\psi(t_{0},s)=\gamma(s)$. Since $f(x,r(t))$ is continuous with respect to $(x,t)$ except for the switching time points $\{t_{j}\}$ and continuously differentiable with respect to $x$, then $\phi(t;t_{0},x_{0},r_{t})$ is continuously differentiable with respect to the initial value $x_{0}$. Let $x_{0}=\gamma(s)$. Then, $w=\frac{\partial \psi}{\partial s}$ is well defined and continuous. By the same algebraic manipulations first presented by \cite{Rus_diB_Son_10}, except for $\{t_{j}\}$, $w(t,s)$ is the Carath\'e{o}dory solution of:
\begin{eqnarray}
\begin{cases}\dot{w}=\frac{\partial f}{\partial \psi}(\psi(t),r(t))w\\
w(t_{0},s)=\gamma'(s)\end{cases}\label{var}
\end{eqnarray}

Consider the sequence $|w(\tilde{t}_{j},s)|_{\chi(\tilde{t}_{j})}$ of Eq. (\ref{var}). For all $s\in[0,1]$, by the extended Coppel inequality (\ref{coppel}) and condition (\ref{mm}), we have
\begin{eqnarray*}
|w(\tilde{t}_{j},s)|_{\chi(\tilde{t}_{j}-)}\le \exp\left[\int_{\tilde{t}_{j-1}}^{\tilde{t}_{i}}\alpha(t)dt\right]
|w(\tilde{t}_{j-1},s)|_{\chi(\tilde{t}_{j-1})}.
\end{eqnarray*}
Using inequality (\ref{eq}) between  $|\cdot|_{\chi(\tilde{t}_{j-1})}$ and $|\cdot|_{\chi(\tilde{t}_{j})}$, gives
\begin{eqnarray*}
&&|w(\tilde{t}_{j},s)|_{\chi(\tilde{t}_{j})}\le\beta_{j}|w(\tilde{t}_{j},s)|_{\chi(\tilde{t}_{j}-)}\\
&\le&\beta_{j}
\exp\left[\int_{\tilde{t}_{j-1}}^{\tilde{t}_{i}}\alpha(t)dt\right]
|w(\tilde{t}_{j-1},s)|_{\chi(\tilde{t}_{j-1})}
\end{eqnarray*}
Then, iterating down to $i<j$, we have
\begin{eqnarray*}
|w(\tilde{t}_{j},s)|_{\chi(\tilde{t}_{j})}&\le&\prod_{k=i+1}^{j}\beta_{k}
\exp\left[\int_{\tilde{t}_{i}}^{\tilde{t}_{i}}\alpha(t)dt\right]
|w(\tilde{t}_{i},s)|_{\chi(\tilde{t}_{i})}.
\end{eqnarray*}
For any $t$, noting that from definition $\tilde{t}_{\tilde{N}(t)+1}\ge t>\tilde{t}_{\tilde{N}(t)}$,  we have
\begin{eqnarray*}
&&|w(t,s)|_{\chi(t)}\le\exp\left[\int_{\tilde{t}_{\tilde{N}(t)}}^{t}\alpha(\tau)d\tau\right]
|w(\tilde{t}_{\tilde{N}(t)},s)|_{\chi(\tilde{t}_{\tilde{N}(t)})}\\
&&\le\exp\left[\int_{\tilde{t}_{\tilde{N}(t)}}^{t}\alpha(\tau)d\tau\right]
\cdot\exp\left(\sum_{i=1}^{\tilde{N}(t)}\log\beta_{i}\right.\\
&&\left.+\int_{t_{0}}^{\tilde{t}_{\tilde{N}(t)}}\alpha(\tau)d\tau\right)
|w(t_{0},s)|_{\chi(t_{0})}.
\end{eqnarray*}
Hence, we obtain
\begin{eqnarray}
|w(t,s)|_{\chi(t)}\le \exp(-c (t-t_{0}))|\gamma^\prime(s)|_{\chi(t_{0})},\label{ineq0.1}
\end{eqnarray}
for all $t>t_{0}+T_{0}$. Since $c$ is independent of $s$, $\delta EAS$ can be derived by following similar arguments as presented by \cite{Rus_diB_Son_10}. In detail, since $|\cdot|_{\chi(t)}$ is uniformly equivalent, there exists $D>0$ such that $|x|_{\chi(t)}\le D |x|_{\chi(t')}$ for any $x\in\R^{n}$ and $t,t'\ge t_{0}$. Then, by (\ref{ineq0.1}), we have
\begin{eqnarray}
&&|\phi(t;x(0),r_{t})-\phi(t;y(0),r_{t})|_{\chi(t_{0})}\\
&\le& D |\phi(t;x(0),r_{t})-\phi(t;y(0),r_{t})|_{\chi(t)}\nonumber\\
&=&D\left|\int_{0}^{1}\frac{\partial \psi(t,s)}{\partial s}ds\right|_{\chi(t)}
\le D\int_{0}^{1}|w(t,s)|_{\chi(t)}ds\nonumber\\
&\le& D\int_{0}^{1}\exp(-c(t-t_{0}))|\gamma^\prime(s)|_{\chi(t_{0})}ds\nonumber\\
&=&D\kappa\exp(-c(t-t_{0}))|x_0-y_0|_{\chi(t_{0})}\label{cr}
\end{eqnarray}
which converges to zero exponentially. This proves that system (\ref{ds1}) is $\delta EAS$.

If (\ref{cond_thm0}) holds for some $c>0$ independently of $t_{0}$, then (\ref{ineq0.1}) holds independently of $t_{0}$ if $t-t_{0}>T_{0}$, which implies $\lim_{t\to\infty}|w(t)|_{\chi(t_{0})}=0$ is uniform with respect to $t_{0}$. By the arguments above, system (\ref{ds1}) is $\delta UEAS$. In addition, inequality (\ref{cr}) implies that the exponential convergence rate can be estimated as $O(\exp(-c(t-t_{0})))$.
\end{proof}

Note that $c$ is an estimate of the exponential convergence rate and also an index of the average contraction rate of (\ref{ds1}).

As a simple example to illustrate this result, define a switched matrix measure based on the following multiple norms: $
|x|_{P(t)}=\sqrt{x^{\top}P(t)x}$.
Here, $P(t)\in\R^{n,n}$ is a differentiable matrix-value function where each $P(t)$ is a symmetric positive definite matrix such that its largest eigenvalue $\lambda_{\max}(P(t))$  the smallest eigenvalue $\lambda_{\min}(P(t))$ is upper bounded and lower-bounded positive respectively. Then, for any $A\in\R^{n,n}$, by simple algebraic manipulations, it can be shown that the induced switched matrix measure is $
\nu_{P(t)}(A)=\frac{1}{2}
\lambda_{\max}[P^{-1/2}(t)(P(t)A+A^{\top}P(t)+\dot{P}(t))P^{-1/2}(t)]$.

Now, consider the linear time-varying (LTV) system $
\dot{x}=A(t)x+I(t)$
with some PWC matrix function $A(t)$. Then, (\ref{cond_thm0}) in Theorem \ref{thm0} can be fulfilled by assuming that it holds that $
P(t)A(t)+A(t)^{\top}P(t)+\dot{P}(t)<0$
for almost every $t$. A similar condition can also be obtained by using the theory of convergent systems \citep{PaPo:07}.

\begin{remark}
Analogous  arguments to those presented by \cite{Loh_Slo_98} and \cite{Rus_diB_Son_10} can be followed to prove that the contracting system converges towards a periodic solution if $r(t)$ is periodic. Indeed, in this case the time instants where $r(t)$ switches, say $\{t_{j}\}$, are also periodic with period $T$ and so is the function $N(t,s)=N(t+T,s+T)$ for all $t>s$. Then, inequality (\ref{cond_thm0}) can be fulfilled by assuming that
\begin{equation}
\frac{1}{T}\left[\int_{t_{0}}^{T+t_{0}}\alpha(t)dt
+\sum_{j=1}^{J}\log(\beta_{j})\right]<-c.
\label{cond_thm2}
\end{equation}
where $J$ is the number of switches in one period. (Note that this inequality holds independently of the initial time $t_{0}$, due to the periodicity.) From Theorem \ref{thm0}, one can conclude that system (\ref{ds1}) is $\delta EAS$. The existence of a periodic solution of (\ref{ds1}) follows by using the contraction mapping theorem.
\end{remark}

\section{Contraction analysis of staircase switching and transaction coefficients between norms}
As a special case and an important application, in this section, we assume that the switching signal $r(t)$ of the switched system (\ref{ds1}) is a staircase function:
\begin{itemize}
\item[${\mathcal H_{4}}$:]$
r(t)=\xi_{i},\quad t_{i}\le t< t_{i+1}$, $i=0,1,\cdots$,
with an increasing time point sequence $\{t_{i}\}$ with $t_{i}=0$ and $\lim_{i\to\infty}t_{i}=\infty$, where $\xi_{i}\in\Xi$ where $\Xi$ is a countable set.
\end{itemize}
Let $\Delta_{i}=t_{i+1}-t_{i}>0$ be the interval between two points, and $N(t,s)=\#\{i:~s<t_{i}<t\}$, be the number of time instants when $r(t)$ switches that fall in the time interval $[s,t]$, in particular, $N(t)=N(t,t_{0})$.

By using multiple norms $|\cdot|_{\chi(t)}$ with $\chi(t)=r(t)$, we obtain the following result as a consequence of Theorem \ref{thm0}.
\begin{corollary}\label{thm1}
Suppose that hypotheses $\mathcal H_{1,2,3,4}$ hold. If there exists constants $\alpha_{i}$, nonnegative constants $\beta_{i}$, $i=1,2,\cdots$, $c>0$ and $T_{0}>0$ such that, for all $i$ and $x\in C$, the following conditions hold
\begin{eqnarray}
\mu_{\xi_{j}}\left(\frac{\partial f}{\partial x}(x,\xi_{j})\right)\le \alpha_{j},\quad |\cdot|_{\xi_{j+1}}\le \beta_{j}|\cdot|_{\xi_{j}},~\forall~j\label{eq1}
\end{eqnarray}
and, for all $T>T_0$,
\begin{eqnarray}
&&\frac{1}{T}\bigg\{\sum_{i=0}^{N(t_{0}+T)}\left[\alpha_{i}\Delta_{i}+\log(\beta_{i})\right]
\nonumber\\
&&+\alpha_{N(t_{0}+T)+1}[t_{0}+T-t_{N(t_{0}+T)}]\bigg\}<-c,
\label{cond_thm1}
\end{eqnarray}
then (\ref{ds1}) is $\delta EAS$ with respect to $r(t)$ in $C$; if (\ref{cond_thm1}) holds for some $c>0$ independently of $t_{0}$, then (\ref{ds1}) is $\delta UEAS$ with respect to $r(t)$ in $C$. Moreover, the exponential convergence rate is estimated as $O(\exp(-c(t-t_{0})))$.
\end{corollary}

\begin{proof}
Let $\alpha(t)=\alpha_{i}$ if $t\in[t_{i},t_{i+1})$. We have
\begin{eqnarray*}
\nu_{r(t)}\left(\frac{\partial f}{\partial x}(x,\xi_{i})\right)=\mu_{r(t)}\left(\frac{\partial f}{\partial x}(x,\xi_{i})\right)\le\alpha(t),~\forall~t\ne t_{i}.
\end{eqnarray*}
Note that
\begin{eqnarray*}
\frac{1}{T}\int_{t_{0}}^{T}\alpha(t)dt&=&\frac{1}{T}\bigg[\alpha_{N(t_{0}+T)+1}(t_{0}+T-t_{N(t_{0}+T)})
+\alpha_{0}\Delta_{0}\\
&&+\sum_{i=0}^{N(t_{0}+T)}\alpha_{i}\Delta_{i}\bigg].
\end{eqnarray*}
Under condition (\ref{cond_thm1}), one can conclude that there exists $T_{1}>0$ such that
\begin{eqnarray}
\frac{1}{T}\left[\int_{t_{0}}^{t_{0}+T}\alpha(t)dt+\sum_{i=0}^{N(t_{0}+T)}\beta_{i}\right]<-c
\label{thm0_recalling}
\end{eqnarray}
for all $T>T_{1}$. By employing Theorem \ref{thm0}, $\delta EAS$ of (\ref{ds1}) can be proved. Furthermore, if (\ref{cond_thm1}) holds independently of the initial time $t_{0}$, and therefore (\ref{thm0_recalling}) holds independently of $t_{0}$, then (\ref{ds1}) is $\delta UEAS$ with respect to $r(t)$ in $C$. The exponential convergence rate can be derived as done in the proof of Theorem \ref{thm0}.
\end{proof}

We can exploit Corollary \ref{thm1} when considering the case where $r(t)$ ($\xi_{i}$) takes values in a finite set $\Omega=\{\onetoK\}$. Specifically, define for any $t>s\ge 0$
\begin{eqnarray*}
\mathcal T_{k}(s,t)&=&\{\tau\in[s,t]:~r(\tau)=k\},\\
\mathcal N_{kl}(s,t)&=&\#\{i:~r(t_{i}-)=k,~{\rm and~}r(t_{i}+)=l,~s\le t_{i}\le t\},
\end{eqnarray*}
and constants $\alpha_{k}$ and $\beta_{kl}>0$, $k,l=\onetoK$ such that
\begin{eqnarray}
\mu_{k}\left(\frac{\partial f}{\partial x}(x,k)\right)\le \alpha_{k},~|\cdot|_{l}\le\beta_{kl}|\cdot|_{k}.\label{ineq1_cor1}
\end{eqnarray}
Then letting $\alpha(t)=\alpha_{i}$ when $t\in[t_{i},t_{i+1})$, it can be seen that inequality (\ref{cond_thm1}) in Corollary \ref{thm1} can be derived if the following condition holds for all $T>T_0$:
\begin{eqnarray}
&&\frac{1}{T}\sum_{k=1}^{K}\left[\alpha_{k}\mathcal T_{k}\left(t_{0},t_{0}+T\right)+\sum_{l=1}^{K}\log(\beta_{kl})\mathcal N_{kl}\left(t_{0},t_{0}+T\right)\right]\nonumber\\
&&\le -c.\label{cond_cor1}
\end{eqnarray}
Alternatively, condition (\ref{cond_cor1}) can be replaced by the assumption that there exists some $T_{1}>0$ such that for each $T_{1}$-length interval, $[nT_{1},(n+1)T_{1}]$,  it holds that
\begin{eqnarray}
&&\frac{1}{T_{1}}\sum_{k=1}^{K}\bigg[\alpha_{k}\mathcal T_{k}\left(nT_{1},(n+1)T_{1}\right)+\nonumber\\
&&\sum_{l=1}^{K}\log(\beta_{kl})\mathcal N_{kl}\left(nT_{1},(n+1)T_{1}\right)\bigg]\le -c,~\forall~n\ge 0.\label{cond_cor12}
\end{eqnarray}

\begin{remark}
From Ineq. (\ref{cond_cor12}), in the case that all $\alpha_{k}<0$, $k=\onetoK$, which implies that all subsystems are contracting and hence incrementally stable, the switched system is incrementally stable if the duration of the mode in which each subsystem is active is sufficiently long.
\end{remark}
More specifically, if all norms $|\cdot|_{k}$ are identical -- and simply denoted by $|\cdot|$ -- then $\beta_{kl}=1$, and inequality (\ref{cond_thm1}) can be fulfilled by simply assuming that
$(1/T)\sum_{k=1}^{K}\alpha_{k}\T(t,t+T)\le-c$
holds for all $t>0$.
Thus, one can see that if $\alpha_{k}<0$ holds for all $k=\onetoK$ (with respect to the same norm), then (\ref{ds1}) is incrementally stable, which coincides with the results by \cite{RuDi:11}.

The conditions of Theorem \ref{thm0} and Corollary \ref{thm1} as well as (\ref{cond_cor1}) depend on two quantities: the matrix measures $\alpha_{k}$ of the Jacobian of the dynamical system for each constant value of $r(t)$, and the transaction coefficients $\beta_{kl}$ between norms. As system (\ref{ds1}) is defined in a finite dimensional Euclidean space, these vector norms are equivalent, i.e., there exist positive constants $\beta_{kl}$ such that $|x|_{l}\le\beta_{kl}|x|_{k}$ holds for all $x\in\R^{n}$. The role of such transaction coefficients will be further investigated below.

For illustration, consider the LTV system
\begin{equation}
\dot x(t)=A(r(t))x(t) + B\label{LTV}
\end{equation}
where $r(t):[0,\infty[ \mapsto C \equiv \{1,2\}$ with  $A=A(1)$ or $A(2)$ and $B\in\R^{2}$. We assume that the linear system is switched between two constant matrices periodically with identical frequency $\varphi_r$. Consider two vector norms $|\cdot|_{1,2}$, corresponding to matrix measures $\mu_{1,2}$ respectively, with transaction coefficients $\beta_{12}$ and $\beta_{21}$.
Corollary 1, specifically condition (\ref{cond_cor12}), yields that the following inequality is a sufficient conditions for $\delta$UEAS of system (\ref{LTV}):
\begin{eqnarray}
&&\frac{1}{2}\left[\mu_{1}(A(1))+\mu_{2}(A(2))+\varphi_r \cdot(\log\beta_{12}
+\log\beta_{21})\right]\nonumber\\
&&<-c\label{cond_LTV}
\end{eqnarray}
for some $c>0$.

We now discuss the relationship between different classes of vector norms in greater detail.

\noindent{\it\bf Quadratic norms.} The quadratic norm is defined as
$|x|_{P}=\sqrt{x^{\top}Px}$ for some positive definite matrix $P$. For another $|x|_{Q}=\sqrt{x^{\top}Qx}$  with symmetric positive definite matrix $Q$. Then, we have
\begin{eqnarray}
|x|_{Q}\le\sqrt{\lambda_{\max}(P^{-1/2}QP^{-1/2})}|x|_{P}\label{quad}
\end{eqnarray}
for all $x\in\R^{n}$, where $\lambda_{\max}(A)$ denotes the largest eigenvalue of a symmetric square matrix $A$.
The proof of this statement is straightforward and omitted here for the sake of brevity.

A a first example, we take
$
A(1)=\begin{pmatrix}
0 & -1\cr 2  & -3 \end{pmatrix}$, $A(2)=\begin{pmatrix}
0 & -11\cr 2  & -33 \end{pmatrix}$.
It is easy to see that, using the matrix measures $\mu_{i}(\cdot)$ induced by the quadratic weighted vector norm $|x|_{\Theta_i}:=|\Theta_i x|_2$, $i=1,2$ respectively, the matrix measures $\mu_{i}\left[A(i)\right]$ with
$$
\Theta_1= \begin{pmatrix}
0.707 & 0.447 \cr 0.707  & 0.894 \end{pmatrix}^{-1}, \quad
 \Theta_2= \begin{pmatrix}
0.998 & 0.322 \cr 0.0618  & 0.947 \end{pmatrix}^{-1}
$$
are negative for $i=1$ and $i=2$. In particular, $\mu_{1}\left[A(1)\right] \leq - 1:=\alpha_1$ and $\mu_{2}\left[A(t,2)\right] \leq -0.6807 :=\alpha_2$.

From the results reported in Sec. \ref{sec:norm}, we get from Ineq. (\ref{quad}) that:
$|x|_{\Theta_1} \leq 1.796|x|_{\Theta_2}$
and
$|x|_{\Theta_2} \leq 1.05 |x|_{\Theta_1}$;
therefore condition (\ref{eq}) remains satisfied with $\beta_{12}:=1.796$ and $\beta_{21}:=1.05$.
Assuming a switching frequency of $\varphi_r=1$ Hz between the two modes, we have $\Delta_i = 1~s$ for all $i$. Therefore (\ref{cond_LTV}) is satisfied and (\ref{LTV}) is incrementally stable in $\R^2$.

The evolution of the norm of the error for two trajectories starting at initial conditions $[0.5, 0.1]$ and $[0.4, 0.1]$ respectively is shown in Fig. \ref{fig:1} when $B=[0\ 0]$ and $B=[1\ 1]$ respectively. In both cases we observe incremental stability as expected. Note that the use of two different weighted norms makes proving incremental stability much simpler than if one common metric had to be used for both modes as required by previous results \cite[e.g.][]{RuDi:11,DiLi:14}.

\begin{figure}
\begin{center}
\includegraphics[width=0.45\textwidth,height=0.2\textwidth]{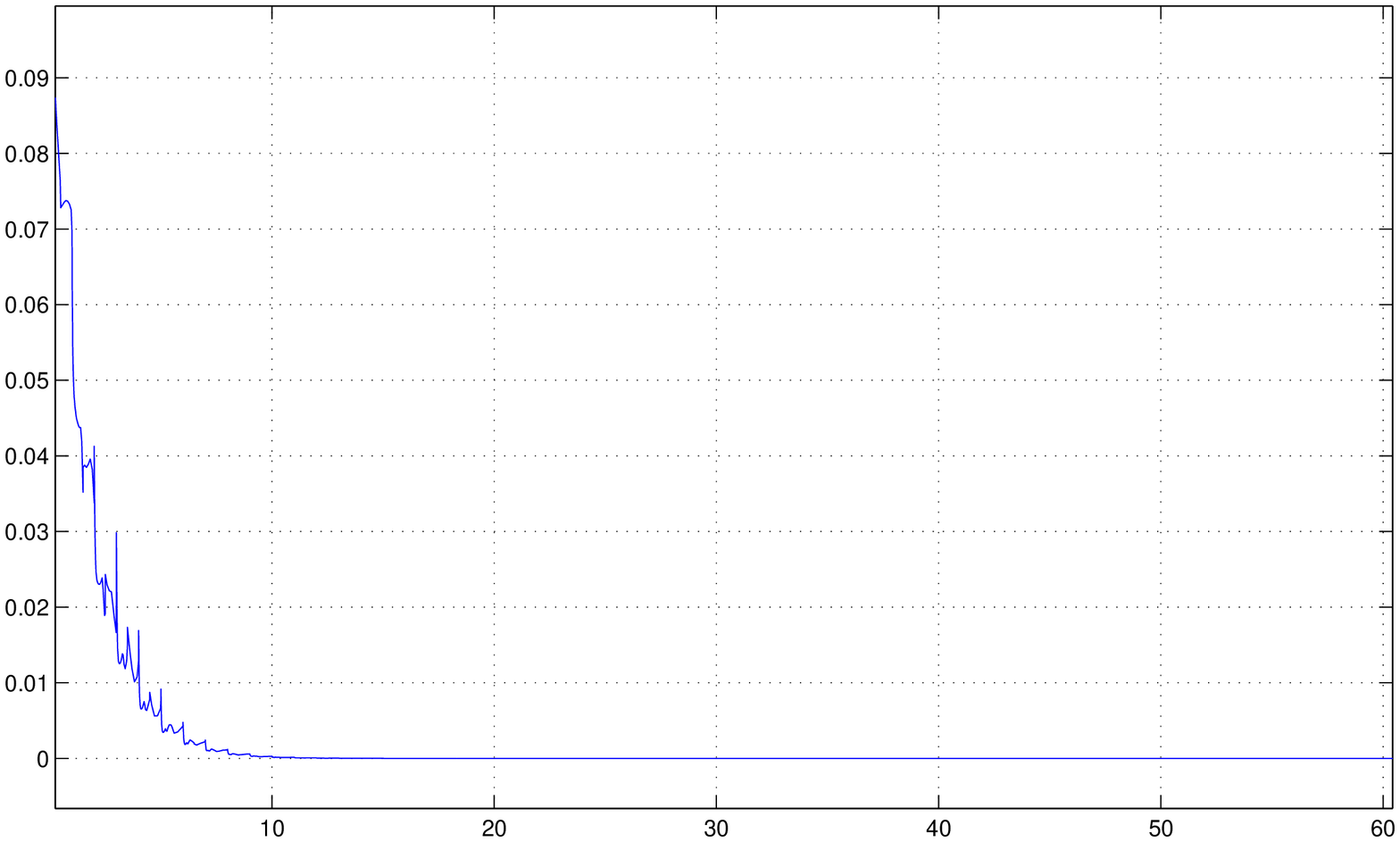}
\includegraphics[width=0.45\textwidth,height=0.2\textwidth]{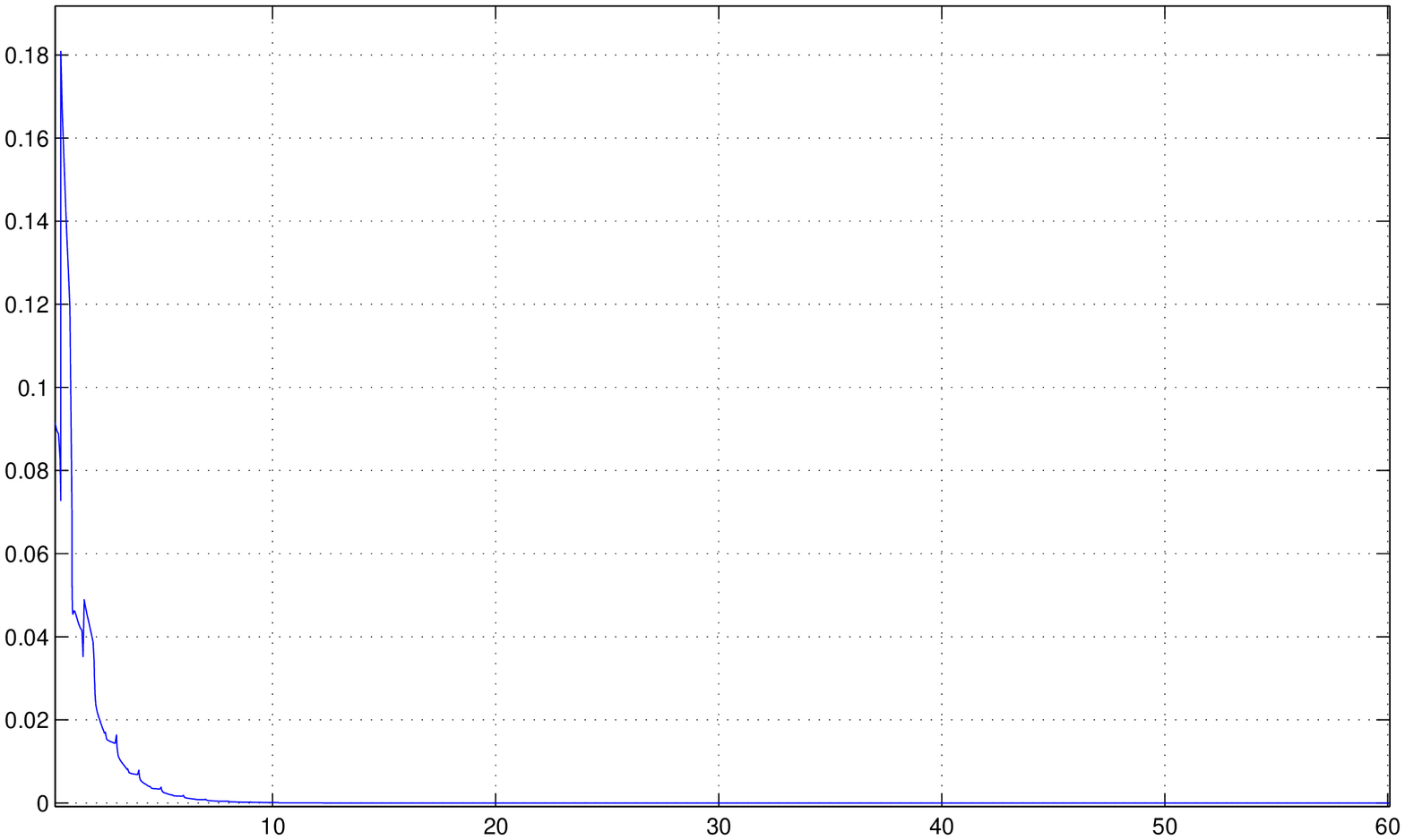}
\caption{Evolution of the Euclidean norm of the error $\|x(t)-y(t)\|$ when $B=[0\ 0]^{\top}$ (top panel) and $B=[1\ 1]^{\top}$ (bottom panel).}\label{fig:1}
\end{center}
\end{figure}


Next, we take $
A(1)=
\begin{pmatrix}
   -1.3481   & -2.9306\cr
    -2.4538  &  -1.2755\end{pmatrix}$, $A(2)=\begin{pmatrix}-11.2237  &  7.0628\cr
    -1.7413 &  1.5119 \end{pmatrix}$.
It can be seen that $A(1)$ is Huwitz stable i.e., with all eigenvalues having negative real parts, but $A(2)$ is unstable, i.e., with some eigenvalues possessing positive real parts. Choosing $\Theta_{1,2}$ as follows:
\begin{eqnarray*}
\Theta_{1}=\begin{pmatrix}
0.3797  &  0.0061\cr
   0.0061   &   0.4534\end{pmatrix} ,\quad \Theta_{2}=\begin{pmatrix}
0.0644 &   -0.1475\cr
      -0.1475   &  0.8267\end{pmatrix}.
\end{eqnarray*}
we find that the matrix measures, $\mu_{1}(A(1))$ and $\mu_{2}(A(2))$, induced by the $2$-norms  $|\cdot|_{\Theta_{1}}$ and $|\cdot|_{\Theta_{2}}$ are:
$\mu_{1}(A(1))=-2.6178$, $\mu_{2}(A(2))=0.9188$.
The transaction coefficients between the two norms $|\cdot|_{\Theta_{1,2}}$ can be calculated by (\ref{quad}) as $|\cdot|_{\Theta_{1}}\le \beta_{12}|\cdot|_{\Theta_{2}}$ and $|\cdot|_{\Theta_{2}}\le \beta_{21}|\cdot|_{\Theta_{1}}$ with $\beta_{12}= 1.9079$ and $\beta_{21}= 10.4207$. The LTV system switches between these two modes with an identical frequency $\varphi_r$ ($\varphi_r=0.25$~Hz in this example). Then (\ref{cond_LTV}) holds for $c=1.1010$. Therefore, the LTV system (\ref{LTV}) is $\delta$UEAS, as shown in Fig. \ref{fig:3}, despite one of its modes being unstable.
\begin{figure}
\begin{center}
\includegraphics[scale=0.2]{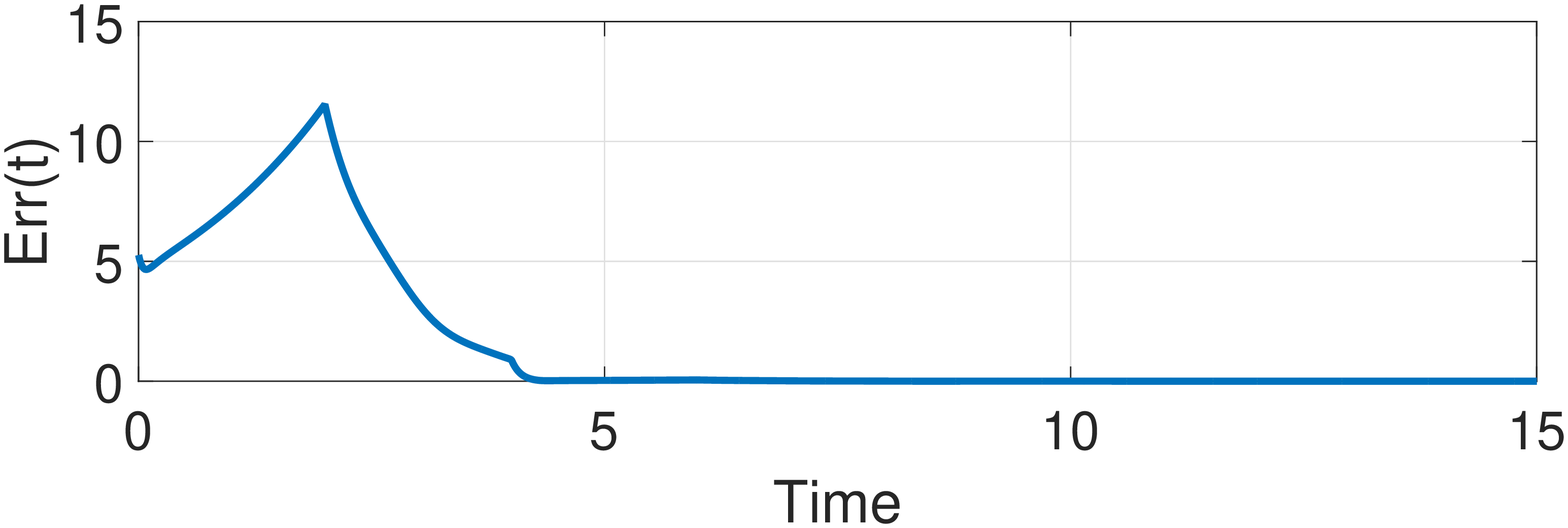}
\caption{Dynamics of the average error $Err(t)$ over $M=10$ independent realizations of random initial values that are picked from $[-10,10]^{2}$ following the uniform distribution. $Err(t)=\frac{1}{M} \sum_{q=1}^{M}\sqrt{(x^{q}_{1}-\bar{x}_{1})^{2}+(x^{q}_{2}-\bar{x}_{2})^2}$ with $\bar{x}_{u}=(1/M)\sum_{q=1}^{M}x_{u}^{q}$ with $u=1,2$, where $x^{q}_{1,2}$ stands for the state components of the $m$-th realization. }\label{fig:3}
\end{center}
\end{figure}

Note $
A_{m}=\frac{A(1)+A(2)}{2}= \begin{pmatrix}-6.2859  &  2.0661\cr
    0.3562  &  0.1182\end{pmatrix}$, which
is unstable as it has eigenvalues of $-6.3988$ and $0.2311$. Hence, the results presented by \cite{PoRo:08} cannot be applied for this specific situation. Nevertheless our extension of contraction analysis to switched systems gives a simple and viable set of conditions that can be used to prove that indeed the system is incrementally stable.

More specifically, from the properties of matrix measures,
\begin{eqnarray*}
&&1/2(\mu(A(1))+\mu(A(2)))\ge\mu(A_{m})\\
&&\ge\max\{\mathcal{Re}(\lambda):~\lambda\in\sigma((A_{m}))\}>0
\end{eqnarray*}
it is not possible to find a uniform norm such that the average of the matrix measures of the switched matrices induced by this matrix norm is negative as required by condition (\ref{cond_LTV}). Therefore, in this case, multiple norms must be utilised for proving contraction and $\delta$EAS of the system.

\noindent{\em\bf Weighted $L_{p}$-norms.}
The weighted $L_{p}$-type norms with $1\le p\le\infty$ are defined as follows.
\begin{itemize}
\item Weighted $L_{p}$-norm: $|x|_{\xi,p}=\left(\sum_{i=1}^{n}\xi_{i}|x_{i}|^{p}\right)^{1/p}$ for some $p\ge 1$ and $\xi=[\xi_{1},\cdots,\xi_{n}]^{\top}$ with $\xi_{i}>0$ for all $i=\oneton$;
\item Weighted $L_{\infty}$-norm: $|x|_{\xi,\infty}=\max_{i}\xi_{i}|x_{i}|$ for some $\xi=[\xi_{1},\cdots,\xi_{n}]^{\top}$ with $\xi_{i}>0$ for all $i=\oneton$.
\end{itemize}
Their transaction coefficients are summarised by the following proposition.
\begin{proposition}\label{prop4} For $p>q\ge 1$ with $p$ possibly equal to $\infty$ and two component-wise vectors $\xi=[\xi_{1},\cdots,\xi_{n}]^{\top}$ and $\eta=[\eta_{1},\cdots,\eta_{n}]^{\top}$ with $\xi_{i},\eta_{i}>0$ for all $i=\oneton$, the following hold: (1) $|x|_{p,\xi}\le\max_{i}\frac{\xi_{i}^{1/p}}{\eta_{i}^{1/q}}|x|_{q,\eta}$; (2) $|x|_{q,\eta}\le\max_{i}\frac{\eta_{i}^{1/q}}{\xi_{i}^{1/p}} (n^{1/q}-n^{1/p})|x|_{p,\xi}$; (3) $|x|_{2,\xi}=|x|_{\Xi}$ with $\Xi={\rm diag}[\xi_{1},\cdots,\xi_{n}]$.
\end{proposition}
This result can be directly derived from \cite{Bou}.
%
Combining \ref{quad} and Proposition \ref{prop4}, we can derive all transaction coefficients between all quadratic norms, $|\cdot|_{Q}$ with positive definite matrix $Q$, and all $|\cdot|_{\xi,p}$ for $+\infty\ge p\ge 1$. The same transaction coefficients hold for the equivalence between the matrix norms induced by these vector norms.

\noindent{\it\bf Structured vector norm.} We define a structured vector norm, following the approach presented in \cite{RuBe:13}. Specifically, assume $x=[x_{1},\cdots,x_{n}]^{\top}\in R^{n}$ can be partitioned into $K$ vectors $x^{k}\in R^{n_{k}}$, $k=\onetoK$, such that $x=[{x^{k}}^{\top},\cdots,{x^{K}}^{\top}]^{\top}$ with $\sum_{k=1}^{K}n_{k}=n$. Let $|x^{k}|_{s_{k}}$  be the norm in $\R^{n_{k}}$.  Then, the {\em structured norm} of $x$ is denoted by $\vert \cdot \vert_G$ and defined as:
\begin{eqnarray}
|x|_{G}=\left|[|x^{1}|_{s_{1}},\cdots,|x^{K}|_{s_{k}}]^{\top}\right|_{S}\label{S}
\end{eqnarray}
where the norm $|\cdot|_{S}$ is defined in $\R^{K}$.

Given the same partition of $x^{k}$, $k=\onetoK$, consider another structured norm $\vert \cdot\vert_{G'}$ based on using the vector norms $|\cdot|_{s_{k}'}$ in $\R^{n_{k}}$, $k=\onetoK$, and $|\cdot|_{S'}$ in $\R^{K}$ such that $
|x|_{G'}=\left|[|x^{1}|_{s_{1}'},\cdots,|x^{K}|_{s_{k}'}]^{\top}\right|_{S'}$.
\begin{proposition}\label{prop5}
Let $\vert \cdot\vert_{G}$ and $\vert \cdot\vert_{G'}$ be two structured norms defined as mentioned above. Let 
$\tau_{S}$ be the transaction coefficient from the norm $|\cdot|_{S'}$ to $|\cdot|_{S}$ and $\tau_{k}$ be the the transaction coefficient from the norm $|\cdot|_{s'_{k}}$ to $|\cdot|_{s_{k}}$. Then, we have $
|x|_{G'}\le \tau_{S}|U|_{S} |x|_{G}$
with $U=diag[\tau_{k}]_{k=1}^{K}$.
\end{proposition}
This result can be derived as a consequence of those presented by \cite{RuBe:13}.

\cite{RuBe:13} showed that contraction analysis can be used to carry out the hierarchical analysis and design of networked systems. Analogously, consider system (\ref{LTV}). Partition $x$ into several sub-vectors, $K$ vectors: $x^{k}\in R^{n_{k}}$, $k=\onetoK$, so that $x=[{x^{k}}^{\top},\cdots,{x^{K}}^{\top}]^{\top}$ with $\sum_{k=1}^{K}n_{k}=n$. Each $x^{k}$ corresponds to the $k$-th subsystem. Then, system (\ref{LTV}) can be equivalently written in the following form:
\begin{eqnarray}
\dot{x}^{k}=\sum_{k'=1}^{K}A_{kk'}(t)x^{k}+B^{k},~k=\onetoK,\label{LTV2}
\end{eqnarray}
where
\begin{eqnarray*}
&&A(r)=\begin{pmatrix}A_{11}(r)&A_{12}(r)&\cdots&A_{1K}(r)\cr
A_{21}(r)&A_{22}(r)&\cdots&A_{2K}(r)\cr
\vdots&\vdots&\cdots&\vdots\cr
A_{K1}(r)&A_{K2}(r)&\cdots&A_{KK}(r)\end{pmatrix},~r=1,2,
\end{eqnarray*}
and $[B_{1}^{\top},\cdots,B^{K}]^{\top}=B$.

The norm of $x$ is defined by (\ref{S}). Let $\tilde{A}_{ij}(t)=|A_{ij}(t)|_{ij}$, where the norm $|\cdot|_{ij}$ is defined by
\begin{eqnarray*}
|A_{ij}(r)|_{ij}=\sup_{|x^{j}|_{s_{j}}=1}|A_{ij}(r)x^{j}|_{s_{i}}.
\end{eqnarray*}
and consider the reduced $K\times K$ matrix
\begin{eqnarray*}
&&\tilde{A}(r)=\begin{pmatrix}\tilde{A}_{11}(r)&\tilde{A}_{12}(r)&\cdots
&\tilde{A}_{1K}(t)\cr
\tilde{A}_{21}(r)&\tilde{A}_{22}(r)&\cdots&\tilde{A}_{2K}(r)\cr
\vdots&\vdots&\cdots&\vdots\cr
\tilde{A}_{K1}(r)&\tilde{A}_{K2}(r)&\cdots&\tilde{A}_{KK}(r)\end{pmatrix},~r=1,2.
\end{eqnarray*}
Let $|\cdot|_{S,1}$ and $|\cdot|_{S,2}$ be two norms in $\R^{K}$ and  $\beta'_{12}$ and $\beta'_{21}$ be their transaction coefficients such that $|\cdot|_{S,1}\le \beta'_{12}|\cdot|_{S,2}$ and $|\cdot|_{S,2}\le \beta'_{21}|\cdot|_{S,1}$. We take the multiple norms in $\R^{n}$ as $
|\cdot|_{G,r}=|\tilde{A(r)}|_{S,r}$, $r=1,2$.
Let $\mu_{G,r}$ be the matrix measure induced by the vector norm $|\cdot|_{G,r}$ in $\R^{n}$ and $\mu_{S,r}$ be that of $|\cdot|_{S,r}$ in $\R^{K}$.  Then, it can be derived that $\mu_{G,r}(A(r))\le\mu_{S,r}(\tilde{A}(r))$ \citep{RuBe:13}.
Proposition \ref{prop5} implies that $\beta'_{12}$ and $\beta'_{21}$ can be the transactions coefficients between $|\cdot|_{G,1}$ and $|\cdot|_{G,2}$ as well.

Therefore, suppose that switched system (\ref{LTV}) is in the block-wise form (\ref{LTV2}). Then,
Corollary \ref{thm1}, specifically inequality (\ref{cond_LTV}), can be fulfilled assuming that it holds that $
(1/2)\left[\mu_{S,1}(\tilde{A}(1))+\mu_{S,2}(\tilde{A}(2))+fr\cdot(\log\beta'_{12}+\log\beta'_{21})
\right]<-c$.

\section{Synchronization in switched networks}
Finally, we consider a network example inspired from one first presented in \cite{DiLi:14}. We assume the network equation is given by
\begin{eqnarray}
\dot{x}^{i}=f(x^{i}(t))-k\sigma(t)\sum_{j=1}^{m}L_{ij}\Gamma x^{j}(t),~i=\onetom.\label{blknet}
\end{eqnarray}
Here, $x^{i}\in\R^{n}$ stands for the state vector at node $i$, $f(\cdot):\R^{n}\to\R^{n}$ the node dynamics, $k$ is the coupling strength, $L=[L_{ij}]_{i,j=1}^{m}$ is the Laplacian matrix associated with a graph $G=[V,E]$, where $V=\{\onetom\}$ is the node set and $E$ the link set, by the way that for each $(i,j)$,  $L_{ij}$ takes value $-1$ if there is a link from node $j$ to $i$ and $0$ otherwise, and $L_{ii}=-\sum_{j=1}^{m}L_{ij}$.  $\sigma(t)$ takes values $0$ or $1$, implying that the diffusive coupling among the nodes in the graph $G$ is only active when $\sigma(t)=1$ while it is not present when $\sigma(t)=0$. $\Gamma\in\R^{n,n}$ stands for the inner coupling matrix.
We can rewrite (\ref{blknet}) in compact form as:
\begin{eqnarray}
\dot{x}=F(x(t))-k\left[\sigma(t)L\otimes\Gamma\right] x(t)\label{blknetm}
\end{eqnarray}
where $x=[{x^{1}}^{\top},\cdots,{x^{m}}^{\top}]^{\top}\in\R^{mn}$,  \\$F(x)=[f^{\top}(x^{1}),\cdots,f^{\top}(x^{m})]^{\top}$ and $\otimes$ is the Kronecker product. Assume that $L$ is diagonalisable, i.e., there exists a nonsingular $Q\in\R^{m,m}$ such that $L=Q^{-1}JQ$ with a diagonal matrix $J=diag[\lambda_{j}]_{j=1}^{m}$, where $\lambda_{j}$, $j=\onetom$, are the eigenvalues of $J$, which are assumed to be real. Without loss of generality, we can assume that $0=\lambda_{1}\ge\lambda_{2}\ge\cdots\ge\lambda_{m}$. Note that $\lambda_{1}=0$ is associated with the synchronization eigenvector $[1,\cdots,1]^{\top}$.

Following similar arguments to those by \cite{Car_Pec_91,RuDi:11,DiLi:14,Yi_Lu_Chen:13}, synchronization of network (\ref{blknetm}) can be achieved if  the following linear systems (obtained via linearisation and block diagonalization of (\ref{blknetm})):
\begin{eqnarray}
\dot{\phi}=[Df(w(t))-k\lambda_{i}\sigma(t)]\phi,~i=2,\cdots,m,\label{syn_var}
\end{eqnarray}
are contracting, where $w(t)$ is a solution of the uncoupled system
\begin{eqnarray}
\dot{w}=f(w),\label{uncoupled}
\end{eqnarray}
and $Df(\cdot)$ is the Jacobian of $f(\cdot)$. Assume
\begin{itemize}
\item[$\mathcal H_{5}$:] System (\ref{uncoupled}) has an asymptotically stable attractor $\mathcal A$ \citep[see][Assumption 2]{Yi_Lu_Chen:13}.
\end{itemize}
\begin{definition}
System (\ref{blknet}) is said to synchronize if there exists $\delta>0$ such that for any $x^{i}(t_{0})\in \mathcal B(\mathcal A,\delta)$, $\lim_{t\to\infty}|x^{i}(t)-x^{j}(t)|=0$ for all $i,j=\onetom$. Here, $\mathcal B(\mathcal A,\delta)=\{y:\inf_{z\in\mathcal A}|y-z|\le \delta\}$ denotes $\delta$-neighbourhood of set $\mathcal A$.
\end{definition}

According to Corollary  \ref{thm1}, in particular, equation (\ref{cond_cor1}), we have the following result:
\begin{proposition}\label{prop_blk}
Consider two vector norms $|\cdot|_{0}$ and $|\cdot|_{1}$, which induce two matrix measures $\mu_{0}(\cdot)$ and $\mu_{1}(\cdot)$ respectively. Suppose that $\mathcal H_{4}$ holds and $\mu_{1}(-\Gamma)\le 0$. If there exist $T_{0}>0$ and $c>0$ such that
\begin{eqnarray}
&&\frac{1}{T}\bigg[\mu_{0}(Df(w))\mathcal T_{0}(nT,(n+1)T)\nonumber\\
&&+\mu_{1}(Df(w)-k\lambda_{2}\Gamma)\mathcal T_{1}(nT,(n+1)T)\nonumber\\
&&+\mathcal N_{01}(nT,(n+1)T)\log\beta_{01}\nonumber\\
&&+\mathcal N_{01}(nT,(n+1)T)\log\beta_{10}\bigg]<-c\label{cond_prop10}
\end{eqnarray}
holds for all $T>T_{0}$ and $w \in \mathcal A$, where $\mathcal T_{u}(s,t)$ stands for the duration of $\sigma(t)=u$ in the time interval $(s,t]$ and $\mathcal N_{uv}(s,t)$ is the number of switches from $\sigma(t)=u$ to $\sigma(t)=v$, for all $u\ne v$, $u,v=0,1$,  then system (\ref{blknetm}) synchronises.
\end{proposition}
\begin{proof}
Note that for each $i=2,\cdots,m$,
\begin{eqnarray*}
&&\mu_{1}(Df(w)-k\lambda_{i}\Gamma)\\&&\le \mu_{1}(Df(w)-k\lambda_{2}\Gamma)+k(\lambda_{i}-\lambda_{2})\mu_{1}(-\Gamma)\\
&&\le \mu_{1}(Df(w)-k\lambda_{2}\Gamma).
\end{eqnarray*}
Due to $\mu_{1}(-\Gamma)\le 0$, then (\ref{cond_prop10}) holds for all $i=2,\cdots,m$.
It can then be proved  in a straightforward manner that the variational equations (\ref{syn_var}) are asymptotically stable for all $i\ge 2$ by verifying condition (\ref{cond_cor1}) and following the same steps of the proof of Theorem 5.1 in \cite{DiLi:14}. Thus, following the arguments in the proof of Theorem 17 in \cite{Yi_Lu_Chen:13}, there is an invariant open convex set $U\in\R^{mn,mn}$ such that $U\supset \mathcal S$, where $\mathcal S=\{x=[{x^{1}}^{\top},\cdots,{x^{m}}^{\top}]^{\top}:x^{i}=x^{j}\in\mathcal A,\forall~i,j\}$, is invariant for the coupled system (\ref{blknetm}). Furthermore,  $\mathcal S$ is an asymptotically stable set for system (\ref{blknetm}). This proves synchronization.
\end{proof}

To illustrate this result, we take
\begin{eqnarray*}
f(w)=\begin{cases} p\{G[-w_{1}+w_{2}]-g(w_{1})\}\\
G[w_{1}-w_{2}]+w_{3}\\
-q w_{2},\end{cases}
\end{eqnarray*}
where $g(w_{1})=m_{0}w_{1}+1/2(m_{1}-m_{0})(|w_{1}+1|-|w_{1}-1|)$ as the Chua' circuit with $m_{0}=-0.5$, $m_{1}=-0.8$, $G=0.7$, $p=9$ and $q=7$, which was reported to exhibit a double-scroll chaotic attractor for the node dynamics (\ref{uncoupled}) \cite{Chua:1985}.
Moreover, we take $\Gamma$ as the identity matrix, and assume the underlining graph of $10$ nodes  has the structure shown in Fig. \ref{fig:4},
\begin{figure}
\begin{center}
\includegraphics[scale=0.2]{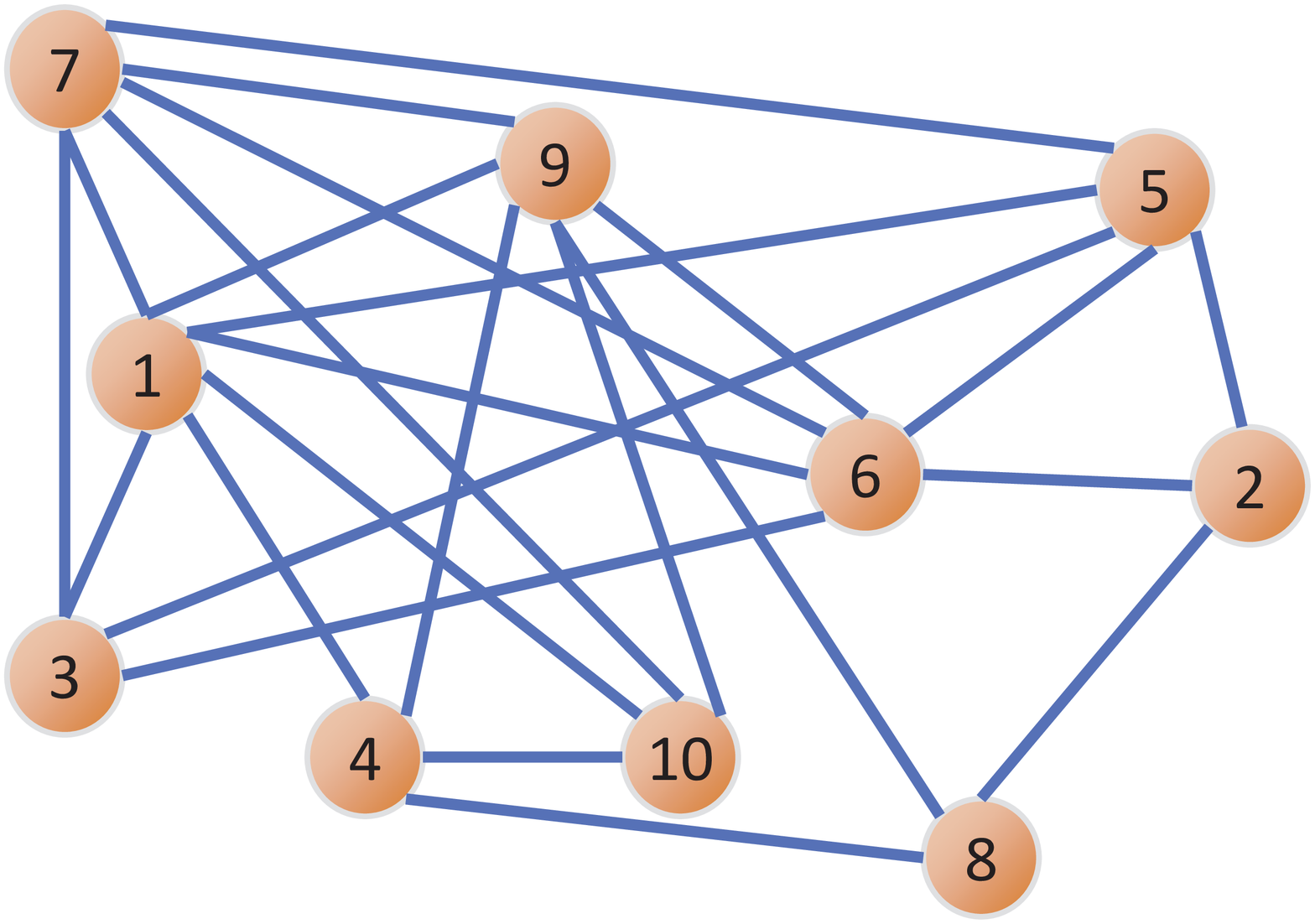}
\caption{The underlying graph topology.}\label{fig:4}
\end{center}
\end{figure}
associated with a Laplacian $L$ with $\lambda_{2}(L)=2.7142$. To study convergence, we choose weighted $1$-norms: $
|y|_{0}=|y_{1}|+r_{2}|y_{2}|+r_{3}|y_{3}|$, with $r_{2}=3.4042$ and $r_{3}=1.0369$, and $|y|_{1}=\sqrt{|y_{1}|^2+|y_{2}|^2+|y_{3}|^2}$,
which implies (i) $\beta_{01}=4.3163$ and $\beta_{10}=1$; (ii) $\mu_{0}(Df(w))\le 3.2829$ for all $t$. We take $k=1$, so that $
\mu_{1}(Df(w)-k\lambda_{2}(L)\Gamma)\le-7.4714$, for all $t$.
The switching signal $\sigma(t)$ is taken as follows:
\begin{eqnarray*}
\sigma(t)=\begin{cases}0&t\in[kT,kT+1/4T)\\
1&t\in(kT+1/4T,(k+1)T)\end{cases}
\end{eqnarray*}
for some $T>0$. Condition (\ref{cond_prop10}) in Proposition \ref{prop_blk} implies that if $T>13.08$, then system (\ref{blknetm}) synchronises. To illustrate this result, we take $T=14$. Fig. \ref{fig:5} shows that, as expected, all nodes synchronise.
\begin{figure}
\begin{center}
\includegraphics[width=0.4\textwidth]{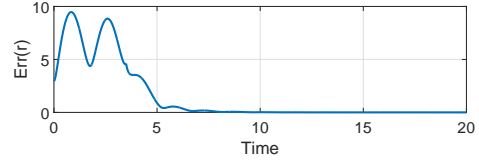}
\caption{Synchronisation dynamics of the network system (\ref{blknetm}) with random initial values that are picked from $[-1,1]^{2}$, $i=1,2,\cdots,10$. Here, $Err(t)=\frac{1}{m} \sum_{i=1}^{m}\sqrt{\sum_{j=1}^{3}(x^{i}_{j}-\bar{x}_{j})^{2}}$ with $\bar{x}_{j}=(1/m)\sum_{i=1}^{m}x_{j}^{i}$.}\label{fig:5}
\end{center}
\end{figure}


\section{Conclusions}
We have presented an extension of contraction analysis to switched Carath\'{e}odory systems. The key step was the definition of switched matrix measures induced by multiple norms. Using these measures, it was possible to derive different sets of sufficient conditions for asymptotic incremental stability of the systems of interest. Most notably, it was possible to prove contraction and hence incremental stability by using different norms, each associated to a different mode of the switched system under investigation. This complements and extends in a highly nontrivial manner to the case of multiple norms, previous results presented by some of the authors \citep[see][]{RuBe:13,DiLi:14}, where contraction was studied by using a common metric.
The theoretical results were illustrated on a set of representative examples and applications showing the effectiveness of the proposed method.

\bibliographystyle{model5-names}
\bibliography{contraction}

\end{document}